\documentclass[10pt]{article}
\usepackage{amsmath,amssymb,amsthm}
\usepackage{amsfonts, amsmath, amsthm, amssymb}
\usepackage{epsfig}
\usepackage{color}
\usepackage[table]{xcolor}
\usepackage[normalem]{ulem}

\title{Signed Magic arrays with certain property}
\author {
  Chanceley Book and  Abdollah Khodkar\\
Department of Computing and Mathematics\\
University of West Georgia\\
Carrollton, GA 30118\\
 {\tt cbook1@my.westga.edu},  {\tt akhodkar@westga.edu}
}

\date{}

\newtheorem{prelem}{{\bf Theorem}}

\newenvironment{lem}{\begin{prelem}{\hspace{-0.5em}{\bf}}}{\end{prelem}}
 \newtheorem{theorem}{Theorem}
\newtheorem{corollary}[theorem]{Corollary}
\newtheorem{lemma}[theorem]{Lemma}

\theoremstyle{definition}

\theoremstyle{remark}

\begin{document}

\maketitle

\begin{abstract}
\noindent A signed magic array, $SMA(m, n;s,t)$,  is an $m \times n$ array with the same number of filled cells $s$ in each row and the same number of filled cells $t$ in each column, filled with a certain set of numbers that is symmetric about the number zero, such that every row and column has a zero sum. We use the notation $SMA(m, n)$ if $m=t$ and $n=s$.

\noindent In this paper, we prove that for every even number $n\geq 2$ there exists an $SMA(m,n)$ such that
the entries $\pm x$ appear in the same row for every $x\in\{1, 2, 3,\ldots, mn/2\}$ if and only if
 $m\equiv 0, 3(\mod4)$ and $n=2$ or $m\geq 3$ and $n\geq 4$.

\noindent {Keywords: magic array, Heffter array, signed magic array, shiftable array}
\end{abstract}

\section{Introduction}\label{introduction}

An {\em integer Heffter array} $H(m, n; s, t)$ is an $m\times n$ array with entries from
$X=\{\pm1,\pm2,$ $\ldots,\pm ms\}$
such that each row contains $s$ filled cells and each column contains $t$ filled cells,
the elements in every row and column sum to 0 in ${\mathbb Z}$, and
for every $x\in X$, either $x$ or $-x$ appears in the array.
The notion of an integer Heffter array $H(m, n; s, t)$ was first defined by Archdeacon in \cite{arc1}.
Integer Heffter arrays with $m=n$ represent a type of magic square where each number from the set
$\{1,2,3,\ldots,n^2\}$ is used once up to sign. A Heffter array is {\em tight} if it has no empty cell; that is,
$n=s$ (and necessarily $m = t$). The proof of the following Theorem can be found in \cite{arc2}.

\begin{theorem}\label{tightHeffter}
Let $m, n$ be integers at least 3.
There is a tight integer Heffter array if and only if $mn\equiv 0, 3 \pmod 4$.
\end{theorem}

For more information on Heffter arrays consult \cite{arc1,arc2,ADDY,DW}.

A {\em signed magic array} $SMA(m,n;s,t)$ is an $m \times n$ array with entries from $X$, where
$X=\{0,\pm1,\pm2, \pm3,\ldots,\pm (ms-1)/2\}$ if $ms$ is odd and $X = \{\pm1,\pm2, \pm3,\ldots,\pm ms/2\}$ if $ms$ is even,
such that precisely $s$ cells in every row and $t$ cells in every column are filled,
every integer from set $X$ appears exactly once in the array and
the sum of each row and of each column is zero.
An $SMA(m,n;s,t)$ is called {\em tight}, and denoted $SMA(m,n)$, if it contains no empty cells; that is $m=t$
(and necessarily $n=s$). Figure \ref {3x2.3x4.4x2} displays tight $SMA(3,2)$, $SMA(3,4)$ and $SMA(4,2)$.

\begin{figure}[ht]
$$\begin{array}{ccc}
    \begin{array}{|c|c|}
	\hline
	1 & -1 \\\hline
	2 & -2 \\\hline
	-3 & 3 \\\hline
	\end{array}&
	\begin{array}{|c|c|c|c|}\hline
	1 & -1 & 2 & -2 \\\hline
	5 & 4 & -5 &-4 \\\hline
	-6 & -3 & 3 & 6 \\\hline
	\end{array}&
\begin{array}{|c|c|}\hline	
       1&-1 \\ \hline	
      -2&2\\ \hline	
      -3&3 \\ \hline
      4&-4\\ \hline\end{array}
\end{array}$$
\caption{An $SMA(3,2)$, $SMA(3,4)$ and $SMA(4,2)$ }
		\label{3x2.3x4.4x2}	
\end{figure}

An $SMA(m, n; s, t)$ is {\em shiftable} if it contains the same number of positive as negative entries in every column and in every row.
These arrays are called \textit{shiftable} because they may be shifted to use different absolute values. By increasing the absolute value of each entry by $k$, we add $k$ to each positive entry and $-k$ to each negative entry. If the number of entries in a row is $2\ell$, this means that we add $\ell k + \ell(-k) = 0$ to each row, and the same argument applies to the columns. Thus, when shifted, the array retains the same row and column sums.
The proof of the following theorem can be found in \cite{KSW}.

\begin{theorem} \label{TH:KSW1}
	An $SMA(m,n)$ exists precisely when $m = n = 1$, or when $m = 2$ and $n \equiv 0, 3 \pmod4$, or when $n = 2$ and $m \equiv 0, 3 \pmod4$, or when $m, n > 2$.
\end{theorem}

\begin{corollary}\label{SMA(m,2)}
Let $n=2$. Then there exists an $SMA(m, 2)$ such that the entries $\pm x$  are in the same row for every
$x\in\{1, 2,3,\ldots, m\}$
 if and only if $m\equiv 0, 3\pmod 4$.
\end{corollary}

We also note that if $A$ is an $m \times n$ tight integer
Heffter array, then the $m\times 2n$ array $[A,-A]$ is an $SMA(m,2n)$ with the property
that the entries $\pm x$ appear in the same row for every $x\in\{1, 2, \ldots, mn\}$.
For more information on signed magic arrays consult \cite{KLE1, KL,AE,KSW}.

In this paper, we prove that for every even number $n\geq 2$ there exists an $SMA(m,n)$ such that
the entries $\pm x$ appear in the same row for every $x\in\{1, 2, 3,\ldots, mn/2\}$ if and only if
 $m\equiv 0, 3(\mod4)$ and $n=2$ or $m\geq 3$ and $n\geq 4$.

For simplicity, we say an $SMA(m,n)$, with $n$ even, has the {\it required property} if
the entries $\pm x$ appear in the same row for every $x\in\{1, 2, 3,\ldots, mn/2\}$.

\section{The case $m$ and $n$ are even}

\begin{theorem}\label {m,n even>=4}
Let $m, n\geq 4$ and even. Then there exists a shiftable $SMA(m,$ $n)$
such that the entries $\pm x$  appear in the same row for  every $x\in\{1, 2, 3,\ldots,$ $(mn/2)\}.$
\end{theorem}

\begin{proof}

Proceed by strong induction first on $n$ and then on $m$. As the base case, we provide arrays for $(m, n) = (4, 4)$, $(6, 4)$, $(4,6)$ and $(6,6)$ in Figures \ref{4x4,6x4} and \ref{4x6,6x6}, respectively.

Now, let $m \in \{4, 6\}$ and $n$ be even, and assume that there exists a shiftable $SMA(m,$ $n - 4)$ with the required property. We may extend this array by adding four columns to create an $m \times n$ array. The empty $m\times 4$ array
may be filled by a shifted copy of the $SMA(4,4)$ or $SMA(6, 4)$ in Figure \ref{4x4,6x4}.

 As the shifted copies each has a row and column sum of zero, they do not change the row sums from the $m \times (n - 4)$ array, and the sums of the new columns will be zero as well. Moreover, the shifted copies have the required property. Therefore, a shiftable $SMA(m, n)$ exists with the required property. Hence, by strong induction on $n$,
a shiftable $SMA(m, n)$ exists for $m \in \{4, 6\}$ and $n \geq 4$ even
such that every  entries $\pm x$  appear in the same row for  every $x\in\{1, 2, 3,\ldots,(mn/2)\}.$
See Figure \ref{6x10} for an illustration.

	Now, let $m$ and $n$ both be even, and $m,n\geq 4$. Assume that there exists a shiftable $SMA(m - 4,n)$ with the required property. We may extend this array by adding four rows to create an $m \times n$ array.
The empty $4\times n$ array
may be filled by a shifted copy of a shiftable $SMA(4,n)$ with the required property, which exists by the above  argument.

Hence, by strong induction on $n$,
a shiftable $SMA(m, n)$ exists for $m, n\geq 4$ and even
such that every  entries $\pm x$  appear in the same row for  every $x\in\{1, 2, 3,\ldots,(mn/2)\}.$
	
\end{proof}

\begin{figure}[ht]
$$\begin{array}{ccc}
\begin{array}{|c|c|c|c|}\hline
	 1&-1 &5&-5 \\ \hline
	 -2&2&-6&6 \\ \hline
	  -3&3&-7&7 \\ \hline
      4&-4&8&-8\\ \hline
\end{array}&&
\begin{array}{|c|c|c|c|}\hline
2&	-2&	6&	-6 \\ \hline
4&	-4&	-8&	8   \\ \hline
-1&	1&	12&	-12  \\ \hline
-3&	3&	-9&	9  \\ \hline
5&	-5&	10&	-10  \\ \hline
-7&	7&	-11&11 \\ \hline
\end{array}
\end{array}$$
\caption{$SSMA(4, 4)$ and $SSMA(6,4)$ with the required property}
\label{4x4,6x4}
\end{figure}

\begin{figure}[ht]
$$\begin{array}{ccc}
\begin{array}{|c|c|c|c|c|c|}\hline
-1&	1&	-5&	5&	-9&	9\\ \hline
2&	-2&	6&	-6&	10&	-10 \\ \hline
3&	-3&	7&	-7&	11&	-11 \\ \hline
-4& 4&	-8&	8&	-12&	12 \\ \hline
\end{array}&&
\begin{array}{|c|c|c|c|c|c|}\hline
1&	-1&	4&	-4&	12&	-12 \\ \hline
2&	-2&	-8&	8&	14&	-14 \\ \hline
3&	-3&	-9&	10&	-10&	9  \\ \hline
-5&	5&	11&	-16&	-11&	16  \\ \hline
6&	-6&	-13&	17&	13&	-17  \\ \hline
-7&	7&	15&	-15&	-18&	18   \\ \hline
\end{array}
\end{array}$$
\caption{$SSMA(4,6)$ and $SSMA(6,6)$ with the required property}
\label{4x6,6x6}
\end{figure}

\begin{figure}[ht]
$$\begin{array}{ccc}
\begin{array}{|c|c|c|c|c|c||c|c|c|c|}\hline
1&	-1&	4&	-4&	12&	-12&     20&	-20&	24&	-24\\ \hline
2&	-2&	-8&	8&	14&	-14& 22&	-22&	-26&	26 \\ \hline
3&	-3&	-9&	10&	-10&	9& -19&	19&	30&	-30  \\ \hline
-5&	5&	11&	-16&	-11&	16& -21&	21&	-27&	27  \\ \hline
6&	-6&	-13&	17&	13&	-17& 23&	-23&	28&	-28  \\ \hline
-7&	7&	15&	-15&	-18&	18& -25&	25&	-29&29   \\ \hline
\end{array}
\end{array}$$
\caption{$SSMA(6,10)$ with the required property obtained by constructions given in Theorem \ref{m,n even>=4}}
\label{6x10}
\end{figure}

\section{The case $m$ odd and $n$ even}

Since the structure of the $SMA(3,n)$ given below is crucial in our constructions, we include the proof of this lemma here which can also be found in \cite{KSW}.

\begin{lemma} \label{3xeven}
Let $n$ be even. Then there exists an $SMA(3,n)$ with the property that the entries $\pm x$  appear in the same row for every $x\in\{1, 2, 3,\ldots,(3n/2)\}$.
\end{lemma}

\begin{proof}
An $SMA(3,2)$ and an $SMA(3,4)$ are given in Figure \ref{3x2.3x4.4x2}.

    Now let $n=2k\geq 6$ and $p_{j}=\lceil\frac{j}{2}\rceil$ for $1\leq j\leq 2k$ . Define a $3\times n$ array $A=[a_{i,j}]$ as follows: For $1\leq j\leq 2k$,

$$a_{1,j}= \begin{cases}
	- \left(\frac{3p_{j}-2}{2}\right) & j \equiv 0 \pmod4 \\
	\frac{3p_{j}-1}{2}& j \equiv 1 \pmod4 \\
	-\left(\frac{3p_{j}-1}{2}\right) & j \equiv 2 \pmod4 \\
	\frac{3p_{j}-2}{2} & j \equiv 3 \pmod4. \\
	
\end{cases}$$
For the third row we define $a_{3,1}=-3k$, $a_{3,2k}=3k$ and when $2\leq j\leq 2k-1$

$$a_{3,j}= \begin{cases}
	- 3(k-p_{j}) & j \equiv 0\pmod4 \\
	3(k-p_{j}+1) & j \equiv 1 \pmod4 \\
	-3(k-p_{j}) & j \equiv 2 \pmod4 \\
	3(k-p_{j}+1) & j \equiv 3 \pmod4. \\
\end{cases}$$
Finally, $a_{2,j}=-(a_{1,j}+a_{3,j})$ for $1\leq j\leq2k$.
It is straightforward to see that array $A$ is an $SMA(3,n)$
with the property that the entries $\pm x$
appear in the same row for every $x\in\{1, 2, 3,\ldots,(5n/2)\}$.
Figure \ref{SMA 3x12} in Appendix 1 displays an $SMA(3, 12)$ constructed by above method.
\end{proof}

\begin{lemma} \label{5xeven}
Let $n\geq 4$ and even. Then there exists an $SMA(5,n)$ with the property that the entries $\pm x$
appear in the same row for every entry $x\in\{1, 2, 3,\ldots,(5n/2)\}$.
\end{lemma}

\begin{proof}
We consider two cases.

\vspace{3mm}
\noindent {\bf Case 1: $n\equiv 0 \pmod 4$:}\quad
Let $A$ be an $SMA(3,n)$ constructed by Lemma \ref{3xeven}.
By construction, the entries in row one of $A$ are:
$$\{\pm(3i+1),\pm(3i+2)\mid 0\leq i\leq(n-4)/4\}\;  \mbox{(See Figure \ref{SMA 3x12} in Appendix 1)}.$$

Switch $3i+1$ with $3i+2$ and $-(3i+1)$ with $-(3i+2)$ for $0\leq i\leq (n-4)/4$ to obtain the $3 \times n$ array
$B$. See Figure \ref{Array 3 by 12} in Appendix 1.
Note that the row sum is still zero in $B$ and the column sums consists of $n/2$ ones and $n/2$ negative ones.

 We now extend array $B$ by adding two rows to create a $5 \times n$ array $C$. The empty $2\times n$ array
may be filled by members of $\{\pm(3n/2)+j\mid 1\leq j\leq n\}$ such that the row sum and column sum of the resulting
$5 \times n$ array are zero. See Figure \ref {SMA(5, 12)} in Appendix 1.


\vspace{3mm}
\noindent {\bf Case 2: $n\equiv 2\pmod 4$}: \quad Figure \ref{SMA(5, 6)}
displays an $SMA(5,6)$ with the required property.

Now let $n\geq 10$ and let $A$ be an $SMA(3,n)$ constructed by Lemma \ref{3xeven}.
By construction, the numbers in row one of $A$ are:

$$\{\pm(3i+1),\pm(3i+2)\mid 0\leq i\leq(n-6)/4\}\cup\{\pm (\frac{3n-2}{4})\}.$$
See Figure \ref{SMA (3x10)}   in Appendix 2.

In row one of $A$ Switch $3i+1$ with $3i+2$ and $-(3i+1)$ with $-(3i+2)$ for $0\leq i \leq (n-10)/4$ and
switch $\frac{3(n-6)}{4}+1$ with   $\frac{3(n-6)}{4}+2$. Note that we switch every entries in row one except
the entries in columns $n-5, n-2, n-1$ and $n$.
Now we switch the entries in row two and columns $n-5, n-2, n-1$ and $n$ as follows:
switch $\frac{-(3n+2)}{4}$ with $\frac{-(3n+10)}{4}$ and $\frac{3n+2}{4}$ with $\frac{3n+10}{4}$
to obtain the $3 \times 10$ array $B$. See Figure \ref{Array 3 by 10} in Appendix 2.

Note that the row sum is still zero in $B$ and the column sums consists of $\frac{n-4}{2}$ ones, $\frac{n-4}{2}$ negative ones, 2 twos and 2 negative twos.

 We now extend array $B$ by adding two rows to create a $5 \times n$ array. The empty $2\times n$ array
may be filled by members of $\{\pm(3n/2)+j\mid 1\leq j\leq n\}$ such that the row sum and column sum of the resulting
$5 \times n$ array, Say C, are zero.
See Figure \ref{MA(5, 10)} in Appendix 2.

\end{proof}

\begin{theorem}\label{maintheorem}
Let $m\geq 3$ be odd and $n$ be even. There exists an $SMA(m,n)$
such that $\pm x$  appear in the same row for $x\in\{\pm1, \pm2,\ldots \pm(mn/2)\},$
if and only if $m\equiv 0, 3\pmod 4$ and $n=2$ or $m\geq 3$ and $n\geq 4$.
\end{theorem}

\begin{proof}
For $n=2$ we apply Corollary \ref{SMA(m,2)}. Now let $m\geq 3$ and $n\geq 4$. We consider two cases.

\noindent {\bf Case 1: $m\equiv 3 (\mod 4)$:} \quad
By Lemma \ref{3xeven} the statement is true for $m=3$. Let $m\geq 7$ and
let $A$ be an $SMA(3,n)$  constructed in Lemma \ref{3xeven}.
We extend array $A$ by adding $m-3$ rows to create an $m \times n$ array. The empty $(m-3)\times n$ array
may be filled by a shifted copy of the $SMA(m-3,n)$ given by Theorem \ref{m,n even>=4}.

 As the shifted copies each have a row and column sum of zero, they do not change the row sums from the $3 \times n$ array, and the sums of the new columns will be zero as well. Moreover, the shifted copies have the required property. Therefore, an $SMA(m, n)$ exists with the property that the entries $\pm x$ appear in the same row for every
 $x\in\{1, 2, 3, \ldots,mn/2\}$.

\vspace{3mm}

\noindent {\bf Case 2: $m\equiv 1 (\mod 4)$:}\quad
By Lemma \ref{5xeven} the statement is true for $m=5$. Now let $m\geq 9$ and let
$A$ be an $SMA(5,n)$ constructed in Lemma \ref{5xeven}.
We extend array $A$ by adding $m-5$ rows to create an $m \times n$ array. The empty $(m-5)\times n$ array
may be filled by a shifted copy of the $SMA(m-5,n)$ given by Theorem \ref{m,n even>=4}.

 As the shifted copies each have a row and column sum of zero, they do not change the row sums from the $5 \times n$ array, and the sums of the new columns will be zero as well. Moreover, the shifted copies have have the required
 property.
 Therefore, an $SMA(m, n)$ exists with the property that the entries $\pm x$ appear in the same row for every
 $x\in\{1, 2, 3, \ldots,mn/2\}$.

\end{proof}


\begin{figure}[ht]
$$\begin{array}{cc}
\begin{array}{|c|c|c|c|}\hline
-9&	9&	7&	-7     \\  \hline
10&	-10&	-8&	8  \\ \hline
1&	-1&	2&	-2     \\ \hline
4&	5&	-4&	-5      \\ \hline
-6&	-3&	3&	6      \\ \hline
\end{array}&
\begin{array}{|c|c|c||c|c|c|}\hline
-14&	12&	2&	14&	-12&	-2 \\ \hline
-6&	-1&	7&	6&	1&	-7          \\ \hline
4&	-13&	9&	-4&	13&	-9      \\ \hline
11&	-8&	-3&	-11&	8&	3        \\ \hline
5&	10&	-15&	-5&	-10& 15     \\ \hline
\end{array}
\end{array}$$
\caption{SMA(5, 4);  $SMA(5,6)$ constructed by a $(5, 3)$ Heffter array and its opposite.}
	\label {SMA(5, 6)}
\end{figure}


\begin{figure}[ht]
	$$\begin{array}{|c|c|c|c||c|c|c|c|}\hline

2&	9&	7&	1&	-2&	-9&	-7&	-1  \\ \hline
-8&	12&	13&	4&	8&	-12&	-13&	-4   \\ \hline
-17&	3&	-16&	-6& 	17&	-3&	16&	6    \\ \hline
18&	-14&	11&	-19&	-18&	14&	-11&	19      \\ \hline
5&	-10&	-15&	20&	-5&	10&	15&	-20    \\ \hline
\end{array}$$
\caption{$SMA(5,8)$ obtained from $H(5, 4)$.}
	\label {SMA(5, 8)}
\end{figure}

\newpage

\centerline{{\bf Appendix 1:} An example for Lemma \ref{5xeven} case 1}

\begin{figure}[ht]
$$\begin{array}{|c|c|c|c|c|c|c|c|c|c|c|c|}\hline
 1 &	-1&	2&	-2&	4	&-4&	5&	-5&	7&	-7&	8&	-8 \\ \hline
 17&	16&	-17&	14&	-16&	13&	-14&	11&	-13&	10&	-11&	-10\\ \hline	
-18&	-15&	15&	-12&	12&	-9&	9&	-6&	6&	-3&	3&	18 \\ \hline
\end{array}$$
 \caption{Array $A:$ $SMA(3, 12)$ constructed by Lemma \ref{3xeven}}
 \label{SMA 3x12}
 \end{figure}

\vspace{5mm}

 \begin{figure}[ht]
$$\begin{array}{|c|c|c|c|c|c|c|c|c|c|c|c|} \hline

2&	-2&	1&	-1&	5&	-5&	4&	-4&	8&	-8&	7&	-7  \\ \hline
  17&	16&	-17&	14&	-16&	13&	-14&	11&	-13&	10&	-11&	-10  \\ \hline
-18&	-15&	15&	-12&	12&	-9&	9&	-6&	6&	-3&	3&	18   \\ \hline
\rowcolor{lightgray}
1&	-1&	-1&	1&	1&	-1&	-1&	1&	1&	-1&	-1&	1 \\ \hline

\end{array}$$
 \caption {B: a $3\times 12$ array obtained from $SMA(3, 12)$, given above, using the construction given in Lemma \ref{5xeven};
  Row 4 displays the column sums.}
 \label{Array 3 by 12}
\end{figure}

\vspace{5mm}

\begin{figure}[ht]
$$\begin{array}{|c|c|c|c|c|c|c|c|c|c|c|c|} \hline

2&	-2&	1&	-1&	5&	-5&	4&	-4&	8&	-8&	7&	-7  \\ \hline
  17&	16&	-17&	14&	-16&	13&	-14&	11&	-13&	10&	-11&	-10  \\ \hline
-18&	-15&	15&	-12&	12&	-9&	9&	-6&	6&	-3&	3&	18   \\ \hline
19&    -19&   -21& 21& 23&-23& -25& 25&   27& -27&-29& 29 \\ \hline
-20&    20& 22&  - 22& -24&24 & 26& -26& -28& 28&30& -30 \\ \hline
\end{array}$$
 \caption {$C$: An SMA(5, 12) obtaine by the construction given in Lemma \ref{5xeven}, Case 1}
 \label{SMA(5, 12)}
\end{figure}

\newpage

\centerline{{\bf Appendix 2}: An example for Lemma \ref{5xeven} case 2}

\begin{figure}[ht]
$$\begin{array}{|c|c|c|c|c|c|c|c|c|c|}\hline
  1 & -1&	2&	 -2&	4	&-4&	5&	-5&	7&	-7 \\ \hline
  14&  13&   -14&  11&  -13&  10&    -11&  8& -10& -8 \\ \hline
 -15& -12&	12&	 -9&	9&	-6&	   6&	-3&	3& 15 \\ \hline
\end{array}$$
 \caption{Array $A:$ $SMA(3, 10)$ constructed by Lemma \ref{3xeven}}
 \label{SMA (3x10)}
 \end{figure}

\vspace{5mm}

\begin{figure}[ht]
$$\begin{array}{|c|c|c|c|c|c|c|c|c|c|}\hline
  2 & -2&	1&	 -1&	 5	&-4&	4&	    -5&	  7&	 -7 \\ \hline
  14&  13&   -14&  11&  -13&  8&   -11&     10&  -8&     -10     \\ \hline
 -15& -12&	12&	 -9&	9&	 -6&	6&	    -3&	  3&     15 \\ \hline
 \rowcolor{lightgray}
 1&-1&-1&1&1&-2&-1&2&2&-2\\ \hline
\end{array}$$

\caption {B: a 3 by 10 array obtained from $SMA(3, 10)$, given above, using the construction given in Lemma \ref{5xeven};  Row 4 displays the column sums.}
 \label{Array 3 by 10}
 \end{figure}

\vspace{5mm}

\begin{figure}[ht]
$$\begin{array}{|c|c|c|c|c|c|c|c|c|c|}\hline
  2 & -2&	1&	 -1&	 5	&-4&	4&	    -5&	  7&	 -7 \\ \hline
  14&  13&   -14&  11&  -13&  8&   -11&     10&  -8&     -10     \\ \hline
 -15& -12&	12&	 -9&	9&	 -6&	6&	    -3&	  3&     15 \\ \hline
 16&   -16& -18& 18& 20& -22&   -20&    22&        23&     -23 \\ \hline
 -17&   17& 19& -19& -21& 24&    21&    -24&        -25&    25 \\ \hline
\end{array}$$
 \caption{ $C$: An SMA(5, 10)obtained by the construction given in Lemma \ref{5xeven}, Case 2}
 \label{MA(5, 10)}
 \end{figure}


\end{document}